\documentclass[11pt]{article}
\usepackage{enumitem}

\usepackage{lineno}

\usepackage{amsfonts,amsmath,amsthm,amscd,amssymb,latexsym,amsbsy,pb-diagram,cite,caption,url}

\usepackage{tikz}
\usetikzlibrary{positioning, calc, trees, decorations.markings, patterns, arrows, arrows.meta}
\usetikzlibrary{trees}

\usepackage{diagmac2}
\usepackage{float}
\usepackage{circledtext}

\newtheorem{theorem}{Theorem}[section]
\newtheorem{lemma}[theorem]{Lemma}
\newtheorem{corollary}[theorem]{Corollary}
\newtheorem{proposition}[theorem]{Proposition}

\newtheorem{conjecture}[theorem]{Conjecture}

\theoremstyle{definition}
\newtheorem{definition}[theorem]{Definition}

\theoremstyle{remark}
\newtheorem{example}[theorem]{Example}
\newtheorem{remark}[theorem]{Remark}

\newcommand{\keywords}{\textbf{Key words. }\medskip}
\newcommand{\subjclass}{\textbf{MSC 2020. }\medskip}
\renewcommand{\abstract}{\textbf{Abstract. }\medskip}

\numberwithin{equation}{section}

\begin{document}

\author{Oleksiy Dovgoshey, Olga Rovenska}

\title{\bf Forbidden Four Cycle, Star Graphs and Isometric Embeddings}

\maketitle

\begin{abstract}
We prove the necessary and sufficient conditions under which
ultrametric spaces of arbitrary infinite cardinality admit isometric embeddings
into ultrametric spaces generated by labeled star graphs.

\end{abstract}

\subjclass{Primary 54E35, Secondary 54E4.}

\keywords{Complete multipartite graph, diametrical graph, discrete ultrametric space,  four-point condition, labeled star graph.}

\section{Introduction}

Over the past several years, ultrametric spaces have been studied using models based on labeled trees. In the finite case, every ultrametric space can be represented up to isometry by tree with non-negative vertex labeling, known as the Gurvich–Vyalyi tree \cite{GV2012DAM}. This representation and its geometric interpretation \cite{PD2014JMS} have led to solutions of various problems in finite ultrametric spaces \cite{DP2019PNUAA,DP2020pNUAA,DPT2015,DPT2017FPTA,Pet2022pANUAA}.
Recently, an analogue of Gurvich--Vyalyi representation was obtained for totally bounded ultrametric spaces \cite{Dovg2025}.

The concept of an ultrametric space generated by an arbitrary non-negative vertex labeling on both finite and infinite trees were introduced in \cite{Dov2020TaAoG} and investigated in \cite{DK2024DLPSSAG,DK2022LTGCCaDUS,DR2025Arxx,DV2025JMS}. König's Infinity Lemma \cite{Konig} shows that infinite star-graphs and rays are fundamental to the construction of any infinite graph. A purely metric characterization of ultrametric spaces generated by labeled star graphs provides in \cite{DR2025USGbLSG}.
Compact ultrametric spaces generated by labeled star graphs are studied in \cite{DCR2025OJAC}. In particular, \cite{DCR2025OJAC} gives a criterion for when a compact ultrametric space admits an isometric embedding into an ultrametric space generated by a labeled star graph. These results set the stage for our work on isometric embeddings of arbitrary ultrametric spaces into those generated by star graphs. 

The main result of the present paper, Theorem~\ref{mnb}, establishes necessary and sufficient conditions for an infinite ultrametric space to admit an isometric embedding into an ultrametric space generated by a labeled star graph.

The proof of Theorem~\ref{mnb} is rather cumbersome, but it is based on the following simple facts. All ultrametric spaces generated by labeled star graphs
are complete and each metric space 
can be isometrically embedded into 
its completion. To prove Theorem~\ref{mnb}, we modify the 
metric of original ultrametric 
space so that the modified 
space contains a Cauchy sequence, 
and show that this 
space is isometrically 
embedded into space generated 
by labeled star graph if and 
only if the original space has 
this property. After that, the 
standard isometric embedding 
of the modified space into its 
completion is used to 
construct the desired isometric 
embedding of the original space. 

The last section of the paper presents several hypotheses that can be probably proven using Theorem~\ref{mnb}.

\section{Preliminaries. Metric spaces}

\hspace{5 mm} Let us denote by $\mathbb{R}^+$ the set $[0, \infty)$.

A \textit{metric} on a nonempty set $X$ is a function $d\colon X\times X\rightarrow \mathbb R^+$ satisfying the following conditions for all 
 \(x\), \(y\), \(z \in X\):
\begin{itemize}[left=10pt]
\item[(i)]  $d(x,y)=d(y,x)$,
\item[(ii)] $(d(x,y)=0)\iff (x=y)$,
\item[(iii)] \(d(x,y)\leq d(x,z) + d(z,y)\).
\end{itemize}

If instead of triangle inequality (iii) the {\it strong triangle inequality}
\begin{equation*}
d(x,y) \leq \max\{d(x,z), d(z,y)\}
\end{equation*}
holds for all $x,y,z \in X$, then $(X,d)$ is called an \emph{ultrametric space}.

\begin{definition}\label{d2.2}
Let \((X, d)\) and \((Y, \rho)\) be metric spaces. A mapping \(\Phi \colon X \to Y\) is called an {\it isometric embedding} of \((X, d)\) in \((Y, \rho)\) if
\[
d(x,y) = \rho(\Phi(x), \Phi(y))
\]
holds for all \(x\), \(y \in X\). In the case when \(\Phi\) is bijective, we say that it is an {\it isometry} of \((X, d)\) and \((Y, \rho)\). The metric spaces are {\it isometric} if there is an isometry of these spaces.
\end{definition}

Let \(S\) be a nonempty subset of a metric space \((X, d)\). The quantity
\begin{equation}\label{e2.1}
\operatorname{diam} S :=\sup\{d(x,y) \colon x, y\in S\}
\end{equation}
is the \emph{diameter} of \(S\). 

We define the \emph{distance set} \(D(X)\) of a metric space \((X,d)\) as
\[
D(X) :=  \{d(x, y) \colon x, y \in X\},
\]
and write 
\begin{equation}
    \label{smwuu}
D_0(X):=D(X)\setminus \{0\}.
\end{equation}
It is clear that \(\operatorname{diam} X = \sup D(X)\). Furthermore, using \eqref{e2.1} and the strong triangle inequality, it is easy to show 
that the equality
\begin{equation}
    \label{lam45}
\operatorname{diam} S = \sup \{ d(x_0,y) \colon y \in S \}
\end{equation}
holds for each nonempty \(S \subseteq X\) and every \(x_0 \in S\) if \((X,d)\) is an ultrametric space.

The following simple proposition seems to be new.  
\begin{proposition}\label{kkigds}
Let $(X,d)$ be an ultrametric space with $|X|\geq 2$.  
If $x_1$ and $x_2$ are two distinct points of $X$ such that
\begin{equation}
    \label{pemmbv1}
d(x_1,x_2) \leq d(y_1,y_2) 
\end{equation}
for all distinct points $y_1,y_2 \in X$,  
then the mapping $\Phi : X \to X$ defined, for every $x\in X$, as  
\begin{equation}
    \label{pemmbv2}
\Phi(x) :=
\begin{cases}
x_2, &  \text{if} \,\,\,x = x_1, \\[6pt]
x_1, &  \text{if} \,\,\,x = x_2, \\[6pt]
x,   & \text{otherwise}
\end{cases} 
\end{equation}
is a self-isometry of the space $(X,d)$.  
\end{proposition}

\begin{proof}
It follows from \eqref{pemmbv2} that $\Phi$ is a self-isometry 
if and only if the equality  
\begin{equation}
    \label{pemmbv3}
d(x,x_1) = d(x,x_2) 
\end{equation}
holds whenever 
\begin{equation}
\label{pemmbv4}
    x_1 \neq x \neq x_2.
\end{equation}
Suppose contrary that there exists $x \in X$ such that \eqref{pemmbv4} holds but  
\[
d(x,x_1) \neq d(x,x_2).
\]
Then, without loss of generality, we may assume
\begin{equation}
\label{pemmbv6}
d(x,x_1) < d(x,x_2). 
\end{equation}
Using \eqref{pemmbv1} with $y_1 = x$ and $y_2 = x_1$, we obtain the inequality
\[
d(x_1,x_2) \leq d(x,x_1).
\]
The last inequality and \eqref{pemmbv6} give us the strict inequality
\[
 \max \{ d(x_2,x_1), d(x_1,x) \}< d(x_2,x),
\]
which contradicts the strong triangle inequality
\[
d(x,x_2) \leq \max \{ d(x,x_1), d(x_1,x_2) \}.
\]
Equality \eqref{pemmbv3} follows.  
Thus $\Phi$ is a self-isometry of $(X,d)$. 

\end{proof}

\begin{corollary}\label{tbhjss}
Let an ultrametric space $(X,d)$ contain two different points $x_1, x_2$ such that \eqref{pemmbv1} holds for all distinct $y_1,y_2 \in X$, and let
$
    X_1 := X \setminus \{x_1\}, $
$
    X_2 := X \setminus \{x_2\}.
$
Then the ultrametric spaces $(X_1, d|_{X_1\times X_1})$ and $(X_2, d|_{X_2\times X_2})$ are isometric.
\end{corollary}

\begin{proof}
Let $\Phi : X \to X$ be the self-isometry of $(X,d)$ defined by \eqref{pemmbv2}.  
It directly follows from \eqref{pemmbv2} that $\Phi(X_1) = X_2$ holds.  
Hence the mapping 
$
    X_1 \ni x \mapsto \Phi(x) \in X_2
$
is an isometry of the spaces $(X_1, d|_{X_1\times X_1})$ and $(X_2, d|_{X_2\times X_2})$.
\end{proof}

\begin{proposition}
    \label{jhfdc}
    Let $(X,d)$ be metric space with $|X| \ge 2$. Then the following statements are equivalent:

\begin{itemize}[left=10pt]
\item[(i)] There exist two distinct points $x_1, x_2 \in X$ such that
\[
d(x_1,x_2) \leq d(y_1,y_2)
\]
for all distinct $y_1,\,y_2\in X$.

\item[(ii)] The set $D_0(X)$ contains the smallest element,
\[
\inf D_0(X) \in D_0(X).
\]
\end{itemize}

\end{proposition}

\begin{proof}
    The validity of the equivalence $(i)\Leftrightarrow (ii)$ follows directly from~\eqref{smwuu}. 
\end{proof}

Let \((X, d)\) be a metric space. An \emph{open ball} with a \emph{radius} \(r > 0\) and a \emph{center} \(c \in X\) is the set
\[
B_r(c) := \{x \in X \colon d(c, x) < r\}.
\]

The next lemma follows from Proposition 18.5 of \cite{Sch1984}.
\begin{lemma}\label{dcfghm}
Let $(X,d)$ be a finite ultrametric space and let $B_1, B_2$ be two different open balls in $(X,d)$. If $B_1 \cap B_2 \neq 0$ holds, then we have either $B_1 \subset B_2$ or $B_2 \subset B_1$. If $B_1$ and $B_2$ are disjoint, then the equality
\[
d(x_1, x_2) = \operatorname{diam}(B_1 \cup B_2)
\]
holds for all $x_1 \in B_1$ and $x_2 \in B_2$.
\end{lemma}

\begin{definition}
\label{alfth}
Let $(X,d)$ be a metric space. If for every 
 $x \in X$ there exists $r > 0$ such that  
\[
\left|B_r(x) \right|= 1,
\]
then
 the metric space $(X, d)$ is called \textit{discrete}.

\end{definition}

Following \cite[p.~48]{Dez-Dez}, we will say that a metric space $(X,d)$ is 
{\it metrically discrete} if there exists $t>0$ such that 
\begin{equation}\label{mytr}
d(x,y) \geq t 
\end{equation}
for all distinct $ x, y \in X.$

It is clear that every metrically discrete space is discrete, but not conversely in general.

The next proposition will be used in Section~5 below.

\begin{proposition}\label{qsswd}
Let $(X,d)$ be a discrete metric space.  
Then the following statements are equivalent:
\begin{itemize}[left=10pt]
    \item[(i)] $(X,d)$ is not metrically discrete.
    \item[(ii)] The equality
    \begin{equation}
        \label{llkk}
        \inf D_0(X) =0
        \end{equation}
    holds, where $D_0(X)$ is defined by~\eqref{smwuu}.
    \item[(iii)] There are a sequence $(x_n)_{n\in\mathbb{N}}$ of distinct points of $X$ 
    and a strictly decreasing sequence $(r_n)_{n\in\mathbb{N}}$ of  positive real numbers 
    such that 
    $$
    \lim\limits_{n\to\infty}r_n=0
    $$
    and
$$
\left|B_{r_n}(x_n)\right|=1,\quad \left|B_{kr_n}(x_n)\right|\geq 1
$$
    for each $n\in {\mathbb N}$ and every $k\in (1,\infty).$
\end{itemize}
\end{proposition}

\begin{proof}
It is easy to see that the inequality
\[
    \inf D_0(X) > 0
\]
implies
\[
    d(x,y) \geq \inf D_0(X) > 0
\]
for all distinct $x,y \in X$.  
Thus~\eqref{mytr} holds with $t = \inf D_0(X)$.  
Conversely, if~\eqref{mytr} holds for all distinct $x,y\in X$ and $t>0$, then we have the inequality
\[
    \inf D_0(X) \geq t > 0.
\]
Therefore the logical equivalence $(i) \Leftrightarrow (ii)$ is valid.

Some simple arguments show that $(iii)$ implies $(i)$.  

To prove the validity of the implication $(i) \Rightarrow (iii)$ we only note that for each $x \in X$,
the inequality
\[
\inf \{ d(x,y) : y \in X \setminus \{x\} \} > 0
\]
holds, and
we have
\[
|B_{k r}(x)| \geq 2
\]
if 
$
r = \inf \{ d(x,y) : y \in X \setminus \{x\} \}
$
and $k \in (1,\infty)$.

\end{proof}

Recall that a sequence $(x_n)_{n \in \mathbb{N}}$ of points of a metric space $(X, d)$ is said to {\it converge} to a point $a \in X$ if
\begin{equation*}
    \lim_{n \to \infty} d(x_n, a) = 0.
\end{equation*}
 A point $x \in X$ is a \textit{limit point}  of a set $A \subseteq X$ if there is a sequence $(a_n)_{n \in \mathbb{N}}$ of different points of $A$ such that $(a_n)_{n \in \mathbb{N}}$ converges to the point $x$.
The set $A$ is said to be {\it dense} in $(X,d)$ if each $x\in X\setminus A$ is a limit point of $A$.

 A sequence $(x_n)_{n \in \mathbb{N}}$ of points of a metric space $(X,d)$ is called a {\it Cauchy sequence} iff
\begin{equation*}
    \lim\limits_{\substack{n \to \infty \\ m \to \infty}} d(x_n, x_m) = 0.
\end{equation*}

We will also use the following ``ultrametric'' form of the concept of Cauchy sequences (see, for example, \cite[p.~4]{PerezGarcia2010} or \cite[Theorem~1.6]{Comicheo2018}).

\begin{proposition}\label{prop:ultrametric-cauchy}
Let $(X,d)$ be an ultrametric space. A sequence $(x_n)_{n\in\mathbb{N}}$ of points of $X$ 
is a Cauchy sequence if and only if the limit relation
\begin{equation*}
\lim_{n\to\infty} d(x_n, x_{n+1}) = 0
\end{equation*}
holds.
\end{proposition}

A metric space $(X,d)$ is \textit{complete} if every Cauchy sequence of points of $X$ converges to a point of $X$.

\begin{definition}\label{ssmmcc}
    Let $(X,d)$ be a metric space.  
A complete metric space $(Y,\rho)$ is called a \textit{completion} of $(X,d)$ if $(X,d)$ is isometric to a dense subspace of $(Y,\rho)$.
\end{definition}


The next proposition directly follows from Definition~2.3 and Theorem~10.12.5 of \cite{Sea2007}.
\begin{proposition}
    \label{fvhjjm}
 Let $(X,d)$ be a metric space, and let $(Y^1,\rho^1)$, $(Y^2,\rho^2)$ be
completions of $(X,d)$. Then $(Y^1,\rho^1)$ and $(Y^2,\rho^2)$  are isometric.
\end{proposition}

The following definition gives us a generalization of the concept of isometry.

\begin{definition}
\label{scguite}
Let $(X, d)$ and $(Y, \delta)$ be metric spaces.
 A bijective mapping $\Phi\colon X \to Y$ is called a {\it weak similarity} if there is a strictly increasing bijection $f\colon D(Y) \to D(X)$ such that the equality
\begin{equation}\label{joi83}
d(x, y) = f \left( \delta \left( \Phi(x), \Phi(y) \right) \right)
\end{equation}
holds for all $x, y \in X$.
\end{definition}

\begin{remark}
The notion of weak similarity was introduced in \cite{DP2013AMH}. Papers \cite{BD2023BKMS,DovBBMSSS2020,DLAMH2020,Pet2018pNUAA} contain  some interconnections between weak similarities and other generalizations of the concept of isometry.
\end{remark}

\section{Preliminaries. Graphs}

\hspace{5 mm} A \textit{graph} is a pair $(V, E)$, where $V$ is a nonempty set and $E$ is a set of unordered pairs $\{u, v\}$ of distinct elements $u, v \in V$. For a graph $G = (V, E)$, the sets $V = V(G)$ and $E = E(G)$ are called the \textit{vertex set} and the \textit{edge set}, respectively. If $\{x, y\} \in E(G)$, then the vertices $x$ and $y$ are called \textit{adjacent}.  A graph is called \textit{finite} if $V(G)$ is finite.

\begin{definition}\label{776dcg}
Let $G_1$ and $G_2$ be graphs. A bijective mapping \linebreak
$\Phi \colon V(G_1)\to V(G_2)$
is called an {\it isomorphism} of $G_1$ and $G_2$ if the equalence 
\begin{equation*}
\{u,v\}\in E(G_1)\iff \{\Phi (u),\Phi (v)\}\in E(G_2)
\end{equation*}
holds for all $u,v\in V(G_1)$. 
The graphs are {\it isomorphic} if there is an isomorphism of these graphs.
\end{definition}

Let $G$ be a graph.
A graph \(G_1\) is a \emph{subgraph} of \(G\) if
\[
V(G_1) \subseteq V(G) \quad \text{and} \quad E(G_1) \subseteq E(G).
\]
In this case we will write \(G_1 \subseteq G\). If $V_1$ is  nonempty subset of $V(G)$, $G_1 \subseteq G$, $V_1 = V(G_1)$  
and $\{u,v\} \in E(G)$ implies
$\{u,v\} \in E(G_1)$ for all $u,v \in V_1$,  
then we say that $G_1$ is an \emph{induced subgraph} of $G$  
and that $G_1$ is induced by $V_1$.

A \emph{path} is a finite graph \(P\) whose vertices can be numbered without repetitions so that
\begin{equation}\label{e3.3-1}
V(P) = \{x_1, \ldots, x_k\} \quad \text{and} \quad E(P) = \{\{x_1, x_2\}, \ldots, \{x_{k-1}, x_k\}\}
\end{equation}
with \(k \geqslant 2\). We will write \(P = (x_1, \ldots, x_k)\) or \(P = P_{x_1, x_k}\) if \(P\) is a path satisfying \eqref{e3.3-1} and said that \(P\) is a \emph{path joining \(x_1\) and \(x_k\)}. A graph \(G\) is \emph{connected} if for every two distinct vertices of \(G\) there is a path \(P \subseteq G\) joining these vertices.

An infinite graph $G$ of the form
\[
V(G) = \{v_1, v_2, \ldots, v_n, v_{n+1}, \ldots \},
\]
\[
E(G) = \{\{v_1,v_2\}, \ldots, \{v_n,v_{n+1}\}, \ldots \},
\]
where $v_i \neq v_j$ for $i \neq j$,
is called a {\it ray}. We say that a graph is {\it rayless} if it contains no rays.

A finite graph $C$ is a \textit{cycle} if $|V(C)|\geq 3$ and there exists an enumeration of its vertices without repetition such that 
\begin{equation}\label{rgjo23}
    V(C) = \{x_1, \ldots, x_n\}, \quad
E(C) = \{\{x_1, x_2\}, \ldots, \{x_{n-1}, x_n\}, \{x_n, x_1\}\}.
\end{equation}
A cycle $C$ satisfying \eqref{rgjo23} is called {\it $n$-cycle}. The $n$-cycles will be denoted by $C_n$ in what follows.

\begin{definition}\label{rfgyppo}
    Let $G $ be a graph and let $k$ be a cardinal number. The graph $G$ is $k$-partite if the vertex set $V (G)$ can be partitioned into $k$ nonvoid disjoint subsets, or parts, in such a way that no edge has both ends in the same part. A $k$-partite graph is complete if any two vertices in different parts are adjacent.
\end{definition}

The complete $k$-partite finite graph with parts $V_1, \ldots, V_k$  will be denoted by
$K_{n_1, \ldots, n_k}$ if we have $n_1 \leq n_2 \leq \ldots \leq n_k$ 
and 
$|V_i| = n_i$ 
for every $i \in \{1,\ldots,n\}$.

The following proposition seems to be well known, but the authors cannot give an exact reference here.

\begin{proposition}\label{zbgqee}
Let $K_{n_1,\ldots,n_k}$ and $K_{m_1,\ldots,m_t}$ be complete
multipartite finite graphs. Then the following statements 
are equivalent:
\begin{itemize}[left=10pt]
    \item[(i)] The equality $k=t$ holds and, in addition, we have 
    $n_i = m_i$ for every $i \in \{1,\ldots,k\}$.
        \item[(ii)] $K_{n_1,\ldots,n_k}$ and $K_{m_1,\ldots,m_t}$ are isomorphic.
\end{itemize}
\end{proposition}

\begin{proof}
It follows directly from Definition~\ref{rfgyppo} that
$K_{n_1,\ldots,n_k}$ and $K_{m_1,\ldots,m_t}$ are isomorphic if statement $(i)$ holds. 

Let $(ii)$ hold. Let $\overline{K}_{n_1,\ldots,n_k}$ and $\overline{K}_{m_1,\ldots,m_t}$ be the complements of $K_{n_1,\ldots,n_k}$ and $K_{m_1,\ldots,m_t}$ respectively. Statement $(ii)$ and Definition~\ref{rfgyppo} imply that $\overline{K}_{n_1,\ldots,n_k}$ and $\overline{K}_{m_1,\ldots,m_t}$ are isomorphic graphs.

Since $K_{n_1,\ldots,n_k}$ and $K_{m_1,\ldots,m_t}$ are complete multipartite graphs, the complements $\overline{K}_{n_1,\ldots,n_k}$ and $\overline{K}_{m_1,\ldots,m_t}$ are disjoint unions
of complete graphs. 
Two complete graphs $G_1$ and $G_2$ are isomorphic if and only if
they have the same number of vertices, $|V(G_1)| = |V(G_2)|$.
Applying this statement to subgraphs of the graphs $\overline{K}_{n_1,\ldots,n_k}$ and
$\overline{K}_{m_1,\ldots,m_t}$ induced by parts of $K_{n_1,\ldots,n_k}$ and, respectively, by parts of $K_{m_1,\ldots,m_t}$, one can easily prove the validity
of statement $(i)$. 
\end{proof}

The following simple corollary of Proposition~\ref{zbgqee} describes the complete multipartite graph which are isomorphic to the cycle $C_4$.

\begin{corollary}\label{renjh}
Let $K=K_{n_1,\ldots,n_k}$ be a complete multipartite finite graph. 
Then the following statements are equivalent:
\begin{itemize}[left=10pt]
    \item[(i)] The equality $k=2$ holds and, in addition, we have  $n_1=n_2=2$.

    \item[(ii)] $K_{n_1,\ldots,n_k}$ is isomorphic to the cycle $C_4$.
\end{itemize} 
\end{corollary}

\begin{proof}
 The cycle $C_4$ admits an enumeration of its vertices such that 
\begin{equation}
    \label{era1}
V(C_4)=\{x_1,x_2,x_3,x_4\}, 
\end{equation}
and
\begin{equation}
    \label{era2}
E(C_4)=\{\{x_1,x_2\},\{x_2,x_3\},\{x_3,x_4\},\{x_4,x_1\}\}.
\end{equation}
Using \eqref{era1} and \eqref{era2} it is easy to prove that $C_4$ is a complete bipartite graph with parts $\{x_1,x_3\}$ and $\{x_2,x_4\}$. Hence the implication $(i)\Rightarrow (ii)$ is valid.  

The validity of $(ii)\Rightarrow(i)$ follows from Proposition~\ref{zbgqee}. 

\end{proof}

As was shown in \cite{DDP2011pNUAA}
the concept of complete multipartite graphs and the concept of
ultrametric spaces are closely related. In order
to describe this relationship, we recall the definition of diametrical graphs.

The following is a modification of Definition 2.1 from \cite{PD2014JMS}.

\begin{definition}\label{d5.2}
Let $(X,d)$ be a metric space. Denote by \(G_{X,d}\) a graph such that \(V(G_{X,d}) = X\) and, for \(u\), \(v \in V(G_{X,d})\),
\begin{equation}\label{d5.2:e1}
(\{u,v\}\in E(G_{X,d}))\Leftrightarrow (d(u,v)=\operatorname{diam} X \text{ and } u \neq v).
\end{equation}
We call $G_{X,d}$ the \emph{diametrical graph} of \((X, d)\).
\end{definition}

\begin{example}\label{ex2.24}
The diametrical graphs of ultrametric spaces depicted in Figure~\ref{cis} are the $C_4$ cycles depicted in Figure~\ref{cis2}.

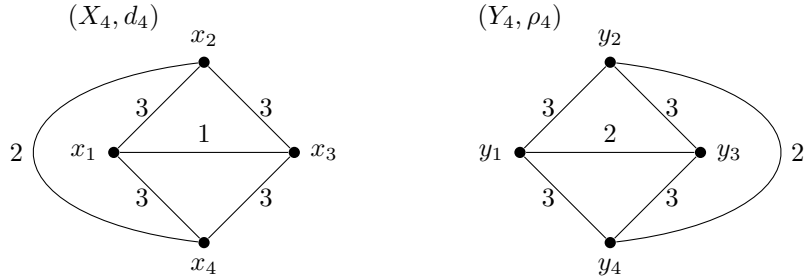
\begin{figure}[H]
\centering
\begin{tikzpicture}[remember picture, scale=1.2, every node/.style={font=\small}]

  \begin{scope}[xshift=-3.5cm]
    \node at (-1.0,1.5) {$(X_4,d_4)$};
    \node[circle,fill,inner sep=1.5pt,label=above:$x_2$] (x2) at (0,1) {};
    \node[circle,fill,inner sep=1.5pt,label=left:$x_1$]  (x1) at (-1,0) {};
    \node[circle,fill,inner sep=1.5pt,label=right:$x_3$] (x3) at (1,0) {};
    \node[circle,fill,inner sep=1.5pt,label=below:$x_4$] (x4) at (0,-1) {};
    \draw (x1) -- node[left] {$3$} (x2);
    \draw (x2) -- node[right,yshift=-0pt] {$3$} (x3);
    \draw (x3) -- node[right] {$3$} (x4);
    \draw (x4) -- node[left] {$3$} (x1);
    \draw (x1) -- node[above] {$1$} (x3);
    \draw (x4) .. controls (-2.5,-0.7) and (-2.5,0.7) .. node[left] {$2$} (x2);
  \end{scope}

  \begin{scope}[xshift=1.0cm]
    \node at (-1.0,1.5) {$(Y_4,\rho_4)$};
    \node[circle,fill,inner sep=1.5pt,label=above:$y_2$] (y2) at (0,1) {};
    \node[circle,fill,inner sep=1.5pt,label=left:$y_1$]  (y1) at (-1,0) {};
    \node[circle,fill,inner sep=1.5pt,label=right:$y_3$] (y3) at (1,0) {};
    \node[circle,fill,inner sep=1.5pt,label=below:$y_4$] (y4) at (0,-1) {};
    \draw (y1) -- node[left] {$3$} (y2);
    \draw (y2) -- node[right,yshift=-0pt] {$3$} (y3);
    \draw (y3) -- node[right] {$3$} (y4);
    \draw (y4) -- node[left] {$3$} (y1);
    \draw (y1) -- node[above] {$2$} (y3);
    \draw (y2) .. controls (2.5,0.7) and (2.5,-0.7) .. node[right] {$2$} (y4);
  \end{scope}

\end{tikzpicture}
\caption{The four-point ultrametric spaces $(X_4,d_4)$ and $(Y_4,\rho_4)$.}
\label{cis}
\end{figure}

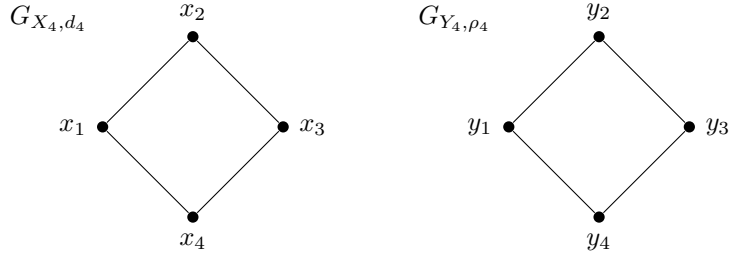
\begin{figure}[H]
\centering
\begin{tikzpicture}[remember picture, scale=1.2, every node/.style={font=\small}]
  \begin{scope}[xshift=-3.5cm]
    \node at (-1.6,1.2) {$G_{X_4,d_4}$};
    \node[circle,fill,inner sep=1.5pt,label=above:$x_2$] (x2) at (0,1) {};
    \node[circle,fill,inner sep=1.5pt,label=left:$x_1$]  (x1) at (-1,0) {};
    \node[circle,fill,inner sep=1.5pt,label=right:$x_3$] (x3) at (1,0) {};
    \node[circle,fill,inner sep=1.5pt,label=below:$x_4$] (x4) at (0,-1) {};
    \draw (x1) -- (x2) -- (x3) -- (x4) -- (x1);
  \end{scope}

  \begin{scope}[xshift=1.0cm]
    \node at (-1.6,1.2) {$G_{Y_4,\rho_4}$};
    \node[circle,fill,inner sep=1.5pt,label=above:$y_2$] (y2) at (0,1) {};
    \node[circle,fill,inner sep=1.5pt,label=left:$y_1$]  (y1) at (-1,0) {};
    \node[circle,fill,inner sep=1.5pt,label=right:$y_3$] (y3) at (1,0) {};
    \node[circle,fill,inner sep=1.5pt,label=below:$y_4$] (y4) at (0,-1) {};
    \draw (y1) -- (y2) -- (y3) -- (y4) -- (y1);
  \end{scope}
\end{tikzpicture}
\caption{The diametrical graphs of $(X_4,d_4)$ and $(Y_4,\rho_4)$.}
\label{cis2}
\end{figure}

\end{example}

The following theorem was proved in \cite{DDP2011pNUAA}.

\begin{theorem}\label{thm:3.1}
Let $(X,\rho)$ be an ultrametric space with $2\leq |X |<\infty$
and let $G_{X,\rho}$ be the diametrical graph of $(X,\rho)$.
Then $G_{X,\rho}$ is a complete $k$-partite finite graph with $k\geq 2$. Conversely, if $G$ is a complete 
$k$-partite finite graph with $k\geq 2$, then there is an ultrametric 
$d \colon V(G) \times V(G) \to \mathbb{R}^+$ such that $G=G_{V,d}$.
\end{theorem}

\begin{remark}
\label{rem:3.3}
Papers~\cite{BDK2022TAG,SDL2022} describe some other interconnections between
complete multipartite graphs and ultrametric spaces.
\end{remark}

\begin{proposition}
\label{prop:3.4}
Let metric spaces $(X,d)$ and $(Y,\rho)$ be finite and weakly similar, 
with $|X| = |Y| \geq 2$.  
Then the diametrical graphs $G_{X,d}$ and $G_{Y,\rho}$ are isomorphic.
\end{proposition}

\begin{proof}
Let $\Phi : X \to Y$ be a weak similarity of $(X,d)$ and $(Y,\rho)$, and let $x,y \in X$. Then, by Definition~\ref{scguite},
  the equality
\begin{equation*}
d(x,y) = \operatorname{diam} X
\end{equation*}
holds if and only if
\begin{equation*}
\rho\bigl(\Phi(x), \Phi(y)\bigr) = \operatorname{diam} Y.
\end{equation*}
Hence, the mapping $\Phi : X \to Y$ is an isomorphism of 
$G_{X,d}$ and $G_{Y,\rho}$ by Definition~\ref{776dcg}.
\end{proof}

\section{Preliminaries. Star graphs and {\bf US}-spaces}

\hspace{5 mm} A connected graph without cycles is called a \textit{tree}.

A \textit{labeled tree} $T(l)$ is a pair $(T, l)$, where $T$ is a tree and $l$ is a function
$
l \colon V(T) \to \mathbb{R}^+.
$

Let $T = T(l)$ be a labeled tree. Following \cite{Dov2020TaAoG}, we consider the mapping $d_l \colon V(T) \times V(T) \to \mathbb{R}^+$,
\begin{equation}
\label{e1.1}
d_l(u, v) :=
\begin{cases}
0, & \text{if } u = v, \\
\max\limits_{w \in V(P)} l(w), & \text{otherwise},
\end{cases}
\end{equation}
where $P$ denotes the unique path connecting $u$ and $v$ in $T$.

\begin{theorem}[\cite{Dov2020TaAoG}, Proposition 3.2]
\label{t1.4}
Let $T = T(l)$ be a labeled tree and $d_l $ be deined by \eqref{e1.1}. Then $d_l$ is an ultrametric on $V(T)$ if and only if
\begin{equation}\label{gjj1ks}
\max\{l(u), l(v)\} > 0
\end{equation}
holds for every edge $\{u, v\} \in E(T)$.
\end{theorem}

A labeling $l \colon V(T) \to \mathbb R^+$ is said to be \textit{non-degenerate} if inequality \eqref{gjj1ks} is satisfied for every edge $\{u,v\}$ of $T$.

The following theorem was proved in \cite{DK2022LTGCCaDUS}.

\begin{theorem}\label{rebv}
    Let $T$ be a tree. Then the following statements are equivalent:

\begin{itemize}[left=10pt]
    \item [(i)] The ultrametric space $(V(T), d_l)$ is complete for every non-degenerate labeling $l \colon V(T) \to \mathbb{R}^+$.

\item[(ii)] $T$ is rayless.
\end{itemize}

\end{theorem}

Let us recall now the concept of star graphs.

\begin{definition}
\label{wer}
A tree $S$ is called a \textit{star graph} if there exists a vertex $c \in V(S)$, referred to as the \textit{center} of $S$, such that $c$ is adjacent to every vertex of the set $V(S) \setminus \{c\}$, and no other edges exist, i.e.,
\begin{equation*}
\{u, w\} \notin E(S) \quad \text{whenever} \quad u \neq c \neq w.
\end{equation*}
\end{definition}

Now we can define the class $\mathbf{US}$ of ultrametric spaces as follows: An ultrametric space $(X, d)$ belongs to $\mathbf{US}$ if there exists a labeled star graph $S(l)$ such that
\begin{equation*}
X = V(S) \quad \text{and} \quad d = d_l,
\end{equation*}
where $d_l$ is defined by \eqref{e1.1} with $T = S$. In this case, we say that $(X, d)$ is an $\mathbf{US}$-space generated by $S(l)$.

The following two results were obtained  in papers \cite{DR2025USGbLSG} and \cite{DCR2025OJAC} respectively.

\begin{theorem}\label{[2.1]}
Let $(X,d)$ be an ultrametric space. Then the following statements are equivalent:
\begin{enumerate}
\item[(i)]  $(X,d) \in {\bf US}$.
\item[(i)]  There is $x_{0} \in X$ such that the inequality
    \begin{equation}\label{zg}
        d(x_{0},x) \leq d(y,x) 
        \end{equation}
    holds whenever
    \begin{equation}\label{cd}     
        x_{0} \neq x \neq y. 
        \end{equation}
\end{enumerate}
\end{theorem}

\begin{proposition}\label{fraas}
Let $(X,d)$ be an ultrametric space and let $x_0 \in X$. If, for $x,y \in X$, inequality \eqref{zg} holds whenever we have \eqref{cd}, then $(X,d)$ is generated by labeled star graph $S(l)$ with the center $x_0$ and the labeling $l: X \to \mathbb{R}^+$ defined as
\begin{equation*}
 l(x): = d(x,x_0)
\end{equation*}
for each $x \in X$.
\end{proposition}

The next two theorems were proved in \cite{DCR2025OJAC}.

\begin{theorem}
\label{eeerlm}
Let $(X,d)$ be an ultrametric space. Suppose that either $X$ is finite, or $X$ has a limit point. Then $(X,d)$ is an $ {\bf US}$-space if and only if $(X,d)$ contains no four-point subspace which is weakly similar to $(X_4,d_4)$ or to $(Y_4,\rho_4)$.
\end{theorem}

\begin{theorem}
\label{8866gh}
Let (X,d) be an {\bf US}-space.
Then $(Y, d|_{Y \times Y})$ is also an ${\bf US}$-space for every finite non-empty $Y \subseteq X$.
\end{theorem}


The original purpose of this paper was to prove the following conjecture, formulated in \cite{DCR2025OJAC}.

\begin{conjecture}
\label{sepjtg}
The following statements are equivalent for every infinite ultrametric space $(X,d)$:
\begin{itemize}[left=10pt]
\item[(i)]  $(X, d) \notin {\bf US}.$

\item[(ii)] $(X, d)$ contains a four-point subspace which is weakly similar either to $(X_4,d_4)$ or to $(Y_4,\rho_4)$.
\end{itemize}
\end{conjecture}

\section{Lemmas and Main result}

\hspace{5 mm}

Let us start from the following proposition.

\begin{proposition}\label{piuu}
Let $(X,d)$ be an $\bf US$--space. Then $(X,d)$ is complete.
\end{proposition}

\begin{proof}
Since $(X,d)\in {\bf US}$ holds, there is a labeled star graph $S(l)$ such that
$
(X,d) = (V(S), d_l).
$
Since $d_l$ is an ultrametric, Theorem~\ref{t1.4} shows that the labeling 
$
l \colon V(S)\to \mathbb{R}^+
$
is non-degenerate.  

Definition~\ref{wer} implies that $S$ is rayless. Hence $(X,d)$ is complete by Theorem~\ref{rebv}.
\end{proof}

The next lemma gives us a rephrasing of Statement $(ii)$ of Conjecture~\ref{sepjtg}.

\begin{lemma}\label{pidscy}
Let $(X,\rho)$ be a four-point ultrametric space. Then the following statements are equivalent:
\begin{itemize}[left=10pt]
    \item[(i)] $(X,\rho)$ is weakly similar either to $(X_4, d_4)$ or to $(Y_4, \rho_4)$.
    
    \item[(ii)] The diametrical graph $G_{X,\rho}$ is isomorphic to $C_4$.
\end{itemize}

\end{lemma}

\begin{proof}
    
$(i) \Rightarrow (ii)$. 
Let $(X,\rho)$ be weakly similar either to $(X_4, d_4)$ or to $(Y_4, \rho_4)$.

It was noted in Example \ref{ex2.24} that the diametrical graphs $G_{X_4, d_4}$ and $G_{Y_4, \rho_4}$ are isomorphic to the cycle $C_4$. Consequently, the diametrical graphs $G_{X,\rho}$ and $C_4$ also are isomorphic by Proposition~\ref{prop:3.4}.

$(ii) \Rightarrow (i)$. 
Let $G_{X,\rho}$ and $C_4$ be isomorphic.
Then there is a numeration of the points of $X$ so that
$G_{X,\rho}$
can be depicted 
by Figure~\ref{cis3}

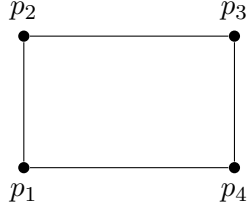
\begin{figure}[H]
  \centering
  \begin{tikzpicture}[scale=0.7]
    \node[fill=black, circle, inner sep=1.5pt, label=below:$p_1$] (p1) at (0,0) {};
    \node[fill=black, circle, inner sep=1.5pt, label=below:$p_4$] (p2) at (4,0) {};
    \node[fill=black, circle, inner sep=1.5pt, label=above:$p_3$] (p3) at (4,2.5) {};
    \node[fill=black, circle, inner sep=1.5pt, label=above:$p_2$] (p4) at (0,2.5) {};

    \draw (p1) -- (p2);
    \draw (p2) -- (p3);
    \draw (p3) -- (p4);
    \draw (p4) -- (p1);
  \end{tikzpicture}
  \caption{The diametrical graph $G_{X,\rho}$.}
  \label{cis3}
\end{figure}

In this case, Definition~\ref{d5.2}
implies the strict inequalities 
\begin{equation}\label{di221}
    \rho(p_1,p_3)<\operatorname{diam} X
\end{equation}
and
\begin{equation}\label{di222}
    \rho(p_2,p_4)<\operatorname{diam} X.
\end{equation}
If the equality 
\begin{equation}\label{wseer}
\rho(p_{1},p_{3})=\rho(p_{2},p_{4}) 
\end{equation}
holds, then the mapping $
\Phi_{1}\colon X \to Y_{4}
$
defined as $
\Phi_{1}(p_{i}) = y_{i},$ $i=1,2,3,4,
$
is a weak similarity of the ultrametric spaces $(X,\rho)$ and $(Y_{4},\rho_{4})$.

Indeed, since $G_{X,\rho}$ and $C_4$ are isomorphic, inequalities \eqref{di221}, \eqref{di222} and the
definition of $(Y_4, \rho_4)$ imply that
\begin{equation*}
|D(X)| = |D(Y_4)| = 3,
\end{equation*}
i.e., the distance sets of $(X,\rho)$ and $(Y_4, \rho_4)$ are three-point subsets of $\mathbb{R}^+$.  
Hence there exists the unique strictly increasing bijective function $f_1 \colon D(Y_4) \to D(X).$
Using Figures~\ref{cis}--\ref{cis3} we can see that equality~\eqref{joi83} holds for all $x,y \in X$ with $(X,d) = (X,\rho),$  $(Y,\delta) = (Y_4, \rho_4),$
and $f = f_1$.  
Thus, by Definition~\ref{scguite}, the mapping $\Phi_1 \colon X \to Y_4$ is a weak similarity of $(X,\rho)$ and $(Y_4, \rho_4)$.

If equality \eqref{wseer} is false, then we have either
\begin{equation}
\label{nabla1}
\rho(p_1, p_3) < \rho(p_2, p_4), 
\end{equation}
or
\begin{equation}
    \label{nabla2}
\rho(p_2, p_4) < \rho(p_1, p_3).
\end{equation}

Suppose that \eqref{nabla1} holds. Then  arguing as in the case of equality~\eqref{wseer}, we can show that the mapping $\Phi_2 \colon X \to X_4,$ defined as  $\Phi_2(p_i) = x_i,$ $i=1,2,3,4,$
is a weak similarity of $(X,\rho)$ and $(X_4, d_4)$.

To complete the proof we only note that the mapping $
\Phi_3 \colon X \to X_4 $
defined by equalities $\Phi_3(p_1) = x_2,$ 
$\Phi_3(p_2) = x_1,$
$\Phi_3(p_3) = x_4,$ 
$\Phi_3(p_4) = x_3$
is the desired
weak similarity of $(X,\rho)$ and $(X_{4},d_{4})$ for the case when inequality~\eqref{nabla2} holds. 

\end{proof}

\begin{lemma}
    \label{fint}
    Let $(X,d)$ be an infinite ultrametric space, $x_0$ and $y_0$ be two distinct fixed points of $X$, and let the following conditions hold:

\begin{itemize}[left=10pt]
\item[(i)] The inequality
\begin{equation}\label{eq1}
d(x_0,y_0) \leq d(p_1,p_2) 
\end{equation}
 holds for all distinct $ p_1,p_2 \in X.$

\item[(ii)] $(X,d)$ contains no four-point subspace whose diametrical graph is isomorphic to the cycle $C_4$.
\end{itemize}

Then $(X,d)$ is an ${\bf US}$-space.

\end{lemma}

\begin{proof}
 Let $(i)$ and $(ii)$ be satisfied.  
It is necessary to prove the membership relation
\begin{equation}\label{eq2}
(X,d) \in {\bf US}.
\end{equation}

 Theorem~\ref{[2.1]} implies that \eqref{eq2} is valid if the inequality 
\begin{equation}\label{dfftgww}
    d(x_0,v)\leq d(u,v)
\end{equation}
holds for all distinct $u,v\in X$. 
If $u = x_0$, then inequality \eqref{dfftgww} evidently holds for each $v \in X$.
Therefore, it is sufficient to consider the case when $u \in X \setminus \{x_0\}$.
Proposition~\ref{kkigds} and condition $(i)$ imply the equality
\[
d(x_0,v) = d(y_0,v)
\]
 for every $v \in X \setminus \{x_0,y_0\}$.
Consequently it is sufficient to prove~\eqref{dfftgww} for the case when $u$ and $v$
are distinct points of $X \setminus \{x_0,y_0\}$.

Let $x_0, y_0, u,$ and $v$ be pairwise distinct, and let $(A,d_A)$ be the ultrametric space with 
\[
A:=\{x_0,y_0,u,v\}, \quad d_A := d\,\big|_{A\times A}.
\]

In order to prove inequality \eqref{dfftgww} we will describe the structure of the diametrical graph $G_{A,d_A}$.
It is clear that  $G_{A,d_A}$ has four vertices.
By Proposition~\ref{zbgqee}, every four-vertex complete multipartite graph is isomorphic to exactly one of the graphs $K_{1,1,1,1},$ $K_{1,1,2},$ $K_{1,3}$ or $K_{2,2}.$

Theorem \ref{thm:3.1} implies that the diametrical graph $G_{A,d_A}$ is a complete multipartite finite graph.
By condition $(ii)$, the diametrical graph $G_{A,d_A}$ and the cycle $C_4$ are not isomorphic. The graph $K_{2,2}$ and the cycle $C_4$ are isomorphic by Corollary 3.1.  Consequently, $G_{A,d_A}$ is isomorphic to one of the graphs
 $K_{1,1,1,1},$ $K_{1,1,2}$ or $K_{1,3}$,
see Figure~\ref{cis4}.

\begin{figure}[H]
  \centering
  \begin{tikzpicture}[
      vertex/.style={circle,draw,minimum size=3pt,inner sep=0pt,fill=black},
      lab/.style={font=\small}
    ]
    \begin{scope}[xshift=-5.5cm]
      \node[vertex] (a2) at (-1,1) {};
      \node[vertex] (a3) at ( 1,1) {};
      \node[vertex] (a4) at ( 1,-1) {};
      \node[vertex] (a1) at (-1,-1) {};
      \draw (a1) -- (a2) -- (a3) -- (a4) -- (a1);
      \draw (a1) -- (a3) (a2) -- (a4);
      \node[lab,below=12pt] at (0,-1.2) {$K_{1,1,1,1}$};
    \end{scope}

    \begin{scope}[xshift=-1cm]
      \node[vertex] (b2) at (-1,1) {};
      \node[vertex] (b3) at ( 1,1) {};
      \node[vertex] (b4) at ( 1,-1) {};
      \node[vertex] (b1) at (-1,-1) {};
      \draw (b1) -- (b2) -- (b3) -- (b4) -- (b1);
      \draw (b1) -- (b3);
      \node[lab,below=12pt] at (0,-1.2) {$K_{1,1,2}$};
    \end{scope}

    \begin{scope}[xshift=3cm]
      \node[vertex] (c1) at (-1,0) {};
      \node[vertex] (c2) at ( 0.9,1.0) {};
      \node[vertex] (c3) at ( 0.9,0.0) {};
      \node[vertex] (c4) at ( 0.9,-1.0) {};
      \draw (c1) -- (c2) (c1) -- (c3) (c1) -- (c4);
      \node[lab,below=18pt] at (0.45,-1) {$K_{1,3}$};
    \end{scope}
  \end{tikzpicture}
  \caption{ $G_{A,d_A}$ is isomorphic to one of the graphs $K_{1,1,1,1}$, $K_{1,1,2}$ or $K_{1,3}$.}
  \label{cis4}
\end{figure}
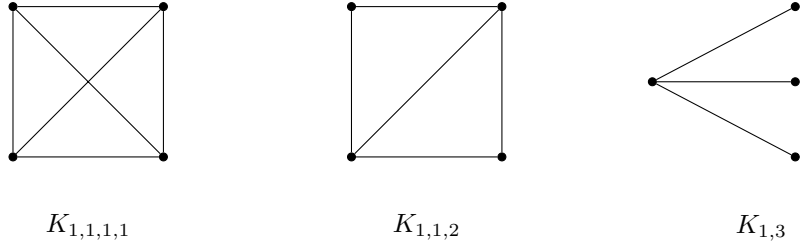

Now we are ready to prove
inequality~\eqref{dfftgww}.

If $G_{A,d_A}$ and $K_{1,1,1,1}$ are isomorphic,
then the distance between
any two distinct points of $(A,d_A)$ is equal to $\operatorname{diam} A$. Thus we have
\begin{equation}\label{eq:diam-eq}
d(x_0,v) = \operatorname{diam} A = d(u,v),
\end{equation}
that implies \eqref{dfftgww}.

If $G_{A,d_A}$ is isomorphic to $K_{1,1,2}$,
then there  is the unique pair of distinct
non-adjacent vertices of $G_{A,d_A}$.
From condition $(i)$ it follows that this pair coincides with $\{x_0,y_0\}$. Consequently the equality 
\begin{equation*}
    d(u,v)=\operatorname{diam}A
\end{equation*}
holds, that implies inequality~\eqref{dfftgww}.

Suppose now that $G_{A,d_A}$ is isomorphic to $K_{1,3}$. Then the ultrametric space $(A,d_A)$ contains 
 a unique point $ p$  such that the equality
\begin{equation}
    \label{dnjjff}
d(p,a) = \operatorname{diam}A 
\end{equation}
holds for each $ a \in A \setminus \{p\}.$ Let us assume that $p\in \{x_0,y_0\}$.
Condition $(i)$ implies the inequality
\[
d(x_0,y_0) \leq d(a_1,a_2)
\]
 for all distinct \(a_1,a_2 \in A\).
 Hence equality \eqref{dnjjff} and Proposition~\ref{kkigds} with \(X = A\) and \(\{x_1, x_2 \}= \{x_0,y_0\}\) give us 
\[
d(x_0,a) = \operatorname{diam}A
\]
for each $ a \in A \setminus \{x_0\}$,
and
\[
d(y_0,a) = \operatorname{diam}A 
\]
for each $ a \in A \setminus \{y_0\}$,
which contradicts the uniqueness of the point \(p \in A \) satisfying~\eqref{dnjjff} for each \(a \in A\setminus \{p\} \).
Thus the membership relation
\[
p \in \{u,v\}
\]
holds. Now using \eqref{dnjjff} with \(p=u,\) \( a=v\) and with \(p=v, \) \(a=u\), we obtain the equality
\[
d(u,v) = \operatorname{diam}A.
\]
The last equality implies inequality~\eqref{dfftgww}.  

Thus, inequality \eqref{dfftgww} is satisfied in all cases.

\end{proof}

The following lemma can be considered as an extension of Theorem~\ref{eeerlm}.

\begin{lemma}\label{gfjhkg}
Let $(X,d)$ be an infinite discrete ultrametric space that satisfies the following conditions:
\begin{itemize}[left=10pt]
\item[(i)] The completion $(Y,\rho)$ of  $(X,d)$ has a limit point $c$.

    \item[(ii)] $(X,d)$ contains no four-point subspace which is weakly similar to $(X_4,d_4)$ or to $(Y_4,\rho_4)$.
\end{itemize}

Then $(Y,\rho)$ is an ${\bf US}$-space.
\end{lemma}

\begin{proof}
 Let $\Phi \colon X \to Y$ be an isometric embedding of
$(X,d)$ in $(Y,\rho)$ such that $\Phi(X)$ is a dense
subset of $(Y,\rho)$.
(This isometric embedding exists according to
Definition~\ref{ssmmcc}).
Let $Y^*$ denote the union of the set $\Phi(X)$
and the singleton $\{c\}$,
\[
Y^* = \Phi(X) \cup \{c\},
\]
and let $\rho^*$ be the restriction
of the ultrametric $\rho : Y \times Y \to \mathbb{R}^+$
on the set $Y^* \times Y^*$,
\[
\rho^* = \rho|_{Y^* \times Y^*}.
\]

It should be noted that $c \notin \Phi(X)$ holds because $(X,d)$ is discrete and, consequently, $\Phi(X)$ is a discrete subset of $(Y,\rho)$. Moreover, $c$ is a limit point of the set $\Phi(X)$ because this set is dense in $(Y,\rho)$ and $c$ is a limit point of $(Y,\rho)$.

We claim that $(Y^*, \rho^*)$ is an $\bf{US}$-space.

Since $(Y^*, \rho^*)$ contains a limit point,  Theorem~\ref{eeerlm} implies that the ultrametric space $(Y^*, \rho^*)$ is an 
${\bf US}$-space if $(Y^*, \rho^*)$ has no four-point subspace which
is weakly similar to $(X_{4}, d_4)$ or $(Y_4, \rho_4)$.

Let $A$ be an arbitrary four-point subset of the set $Y^*$ and let $\rho^*_A = \rho^*|_{A \times A}$. 

If $c\notin A$, then condition $(ii)$ implies that $(A,\rho^*_A)$ is not weakly similar to $(X_4,d_4)$ or to $(Y_4,\rho_4)$.

Let us consider the case when $c\in A$.

Since \(c\) is a limit point of the set \(\Phi(X)\) in the space \((Y^*, \rho^*)\), 
there exists a point \(p \in Y^*\) such that 
\begin{equation}
    \label{thkq2}
p \notin A 
\end{equation}
and 
\begin{equation}
    \label{thkq21}
\rho^*(c,p) <  \rho^*(a_1,a_2)
\end{equation}
holds for all distinct $a_1,a_2\in A$.
Let us denote by \(A_p\) the set \(A \cup \{p\}\) and consider the ultrametric space
\((A_p, \rho^*_{A_p})\) in which the ultrametric \(\rho^*_{A_p}\) is defined as
\[
\rho^*_{A_p} = \rho^* \big|_{A_p \times A_p}.
\]

Let  $r \in (0,\infty)$ satisfy the double inequality
\begin{equation}
    \label{unfds1}
\rho^*(c,p) < r < \min\{ \rho^*(a_1,a_2) : a_1 \ne a_2, \; a_1,a_2 \in A \}.
\end{equation}
Then, using Lemma \ref{dcfghm}  and \eqref{unfds1} we can prove that the inequality
\begin{equation}
    \label{unfds2}
\rho^*(c,p) \leq \rho^*(x_1,x_2) 
\end{equation}
holds for all different $x_1, x_2 \in A_p$.
Let us denote by $A^0_p$ the set $(A \setminus \{c\}) \cup \{p\}$ and consider the ultrametric space $(A_p^0,\rho^0_{A_p})$
with
\[
\rho^*_{A_p^0} = \rho^* \big|_{A_p^0 \times A_p^0}.
\]
Since \eqref{unfds2} holds for all different $x_1, x_2 \in A_p$, the conditions of Corollary \ref{tbhjss} are satisfied for
\[
(X,d) = (A_p, \rho^*_{A_p}), \quad x_1 = p, \; x_2 = c.
\]
Consequently the spaces $(A, \rho^*_A)$ and $(A_p^0, \rho^*_{A_p^0})$ are isometric.

Thus, by condition $(ii)$, the space $(Y^*, \rho^*)$ contains no four-point subspaces which are weakly similar to $(X_4,d_4)$ or to $(Y_4,\rho_4)$.
Since $c$ is a limit point of the space $(Y^*, \rho^*)$, it is an {\bf US}-space by Theorem~\ref{eeerlm}.

Proposition \ref{piuu} implies that $(Y^*,\rho^*)$ is complete. Hence $(Y^*,\rho^*)$ also is a completion of the ultrametric space $(X,d)$. By Proposition~\ref{fvhjjm} the spaces $(Y^*,\rho^*)$ and $(Y,\rho)$ are isometric. Thus $(Y,\rho)$ is an ${\bf US}$-space as required. 

\end{proof}

\begin{lemma}\label{476}
Let an infinite discrete ultrametric space $(X,d)$ satisfy the following conditions:
\begin{itemize}[left=10pt]
    \item[(i)] $(X,d)$ is not metrically discrete.

        \item[(ii)]  $(X,d)$ contains no four-point subspace whose diametrical graph is isomorphic to the cycle $C_4$.

Then $(X,d)$  contains a Cauchy sequence $(x_n)_{n\in \mathbb N}$ of distinct points of $X$.
\end{itemize}

\end{lemma}

\begin{proof}
Since $(X,d)$ is not metrically discrete, Proposition~\ref{qsswd} implies that there are a sequence $(x_n)_{n\in\mathbb{N}}$ of distinct points of $X$ 
    and a strictly decreasing sequence $(r_n)_{n\in\mathbb{N}}$ of  positive real numbers 
    such that 
    \begin{equation}\label{st1}
    \lim\limits_{n\to\infty}r_n=0
        \end{equation}
    and
\begin{equation}
    \label{st2}
\left|B_{r_n}(x_n)\right|=1,\quad \left|B_{kr_n}(x_n)\right|\geq 1
    \end{equation}
    for each $n\in {\mathbb N}$ and every $k\in (1,\infty).$

We claim that $(x_n)_{n\in\mathbb{N}}$ is a Cauchy sequence in $(X,d)$.  
By Proposition~\ref{prop:ultrametric-cauchy} the sequence $(x_n)_{n\in\mathbb{N}}$ is a Cauchy sequence if and only if the limit relation
\begin{equation*}
\lim_{n\to\infty} d(x_n,x_{n+1})=0
\end{equation*}
holds.
The last equality is satisfied if and only if
\begin{equation}\label{xx2}
    \limsup_{n\to\infty} d(x_n, x_{n+1}) = 0.
\end{equation}
Let us assume that \eqref{xx2} is false.  
Then there exist $c_1>0$ and a subsequence $(x_{n_m})_{m\in\mathbb{N}}$ of the sequence $(x_n)_{n\in\mathbb{N}}$ such that
\begin{equation}
    \label{tvhukp}
      d(x_{n_m}, x_{n_m+1}) > c_1
\end{equation}
for every $m\in \mathbb N$.
Now using \eqref{st1} and \eqref{tvhukp} we can find $m_0\in \mathbb N$ such that 
\begin{equation}\label{cc1}
    r_{n_{m_0}+1}< r_{n_{m_0}}<c_1.
\end{equation}
Therefore, there exist $k\in (1,\infty)$ for which the double inequality
\begin{equation}
    \label{cc2}
kr_{n_{m_0}+1}<kr_{n_{m_0}}<c_1
\end{equation}
is satisfied.

For convenience, we introduce the following notation:
\begin{equation}
    \label{cc3}
    x^0 := x_{n_{m_0}}, \quad x^1 := x_{n_{m_0}+1}, \quad
    r^0 := kr_{n_{m_0}}, \quad r^1 := kr_{n_{m_0}+1}.
\end{equation}

Let us consider the open balls $B_{r^0}(x^0)$ and $B_{r^1}(x^1)$.  

Inequality \eqref{tvhukp} with $m=m_0$, double inequality \eqref{cc2}  with $x_m=x_{n_m} = x^0$, and Lemma~\ref{dcfghm} with 
    $B_1 = B_{r^0}(x^0),$ $ B_2 = B_{r^1}(x^1),$
 give us the equality
\begin{equation}
    \label{kkjuuh1}
    B_{r^0}(x^0) \cap B_{r^1}(x^1) = \varnothing.
\end{equation}
The second inequality in \eqref{st2} implies the existence of some points
$
y^0 \in B_{r^0}(x^0),$ and $ y^1 \in B_{r^1}(x^1)
$
which satisfy conditions
\begin{equation}
\label{kkhh4}
y_0 \neq x^0, \qquad y^1 \neq x^1. 
\end{equation}
Now using \eqref{kkjuuh1} and \eqref{kkhh4} we see that the point $x^0,$ $y^0,$ $x^1$ and $y^1$ are pairwise distinct.
Let us denote by $A$ the set $\{x^0, y^0, x^1, y^1\}$ and consider the ultrametric space $(A, d_A)$ with
\[
d_A := d\big|_{A \times A}.
\]

By Theorem \ref{thm:3.1} the diametrical graph $G_{A,d_A}$ is a complete multipartite graph.  
We claim that $G_{A,d_A}$ and $K_{2,2}$ are isomorphic.

To prove it we note that the equalities
\begin{equation}
    \label{865gny}
\operatorname{diam} A=d(x^0, x^1) = d(x^0, y^1) = d(y^0, x^1) = d(y^0, y^1)
\end{equation}
hold by Lemma \ref{dcfghm}.
Moreover, inequality \eqref{tvhukp} with $m=m_0$, double inequality~\eqref{cc1}, 
and equalities~\eqref{cc3} give us the inequalities
\begin{equation}
    \label{vnhgefk}
d(y^0,x^0) <\operatorname{diam} A
\end{equation}
and
\begin{equation}
\label{eftyi88}
d(y^1,x^1) <\operatorname{diam} A. 
\end{equation}
Now using \eqref{865gny}--\eqref{eftyi88} , we see that $G_{A,d_A}$ and $K_{2,2}$ are isomorphic.  

Hence $G_{A,d_A}$ and $C_4$ also are isomorphic by Corollary~\ref{renjh}.  
The last statement contradicts condition $(ii)$.  

Thus $(x_n)_{n \in \mathbb{N}}$ is a Cauchy sequence in $(X,d)$.

\end{proof}

The following theorem shows that Conjecture~\ref{sepjtg} is true.

\begin{theorem}\label{mnb}
Let $(X,d)$ be an infinite ultrametric space. The following statements are equivalent:
\begin{itemize}[left=10pt]
\item[(i)] There exists an ultrametric space $(Y,\rho)\in {\bf US}$ such that $(X,d)$ is isometric to a subspace of $(Y,\rho)$.
\item[(ii)] $(X,d)$ contains no four-point subspace whose diametrical graph is isomorphic to the cycle $C_4$.

\item[(iii)] $(X,d)$ contains no four-point subspace which is weakly similar to $(X_4,d_4)$ or to $(Y_4,\rho_4)$.

\end{itemize}
\end{theorem}

\begin{proof}

The logical equivalence $(ii) \Leftrightarrow (iii)$ is valid by Lemma~\ref{piuu}.  
Theorems \ref{eeerlm} and \ref{8866gh} imply the validity of the implication $(i) \Rightarrow (ii)$.  

Thus to prove the theorem under consideration we must show that statement $(i)$ is true if statements $(ii)$ and $(iii)$ are true.  
Let us do it.

Suppose that statements $(ii)$ and $(ii)$ are valid.  Let us denote by $\Delta$ the infimum of the set 
\begin{equation}
    \label{ety1}
    D_0(X) := \{ d(x,y) : x,y \in X, \; x \neq y \}, 
\end{equation}
\begin{equation}
\label{ety2}
    \Delta:=\inf D_0(X).
\end{equation}
The following three cases are posible:
\begin{itemize}[left=10pt]
    \item[$(i_1)$] $\Delta=0$,

       \item[$(i_2)$] $\Delta>0$ and $\Delta \in D_0(X)$,

   \item[$(i_3)$] $\Delta>0$ and $\Delta \notin D_0(X)$.
    
\end{itemize}

Below we will prove the truth of the implication $((ii)\&(iii))\Rightarrow (i)$ in each of the cases $(i_1)$--$(i_3)$ separately.

$(i_1)$ Let $(i_1)$ hold.

If $(X,d)$ has a limit point, then $(X,d)$ is an $\bf US$-space by Theorem~\ref{eeerlm}. Hence statement $(i)$ holds with 
\((Y,\rho) = (X,d)\). If \((X,d)\) contains no limit points, then, using 
Definition~\ref{alfth} and Proposition~\ref{qsswd}, 
we can prove that \((X,d)\) is discrete 
but not metrically discrete.
Consequently $(X,d)$ contains a Cauchy sequence $(x_{n})_{n\in\mathbb{N}}$ of distinct points of $X$ by Lemma~\ref{476}.  
Let $(Y,\rho)$ be the completion of  $(X,d)$ and let $\Phi\colon X\to Y$ be an isometric embedding of $(X,d)$ in $(Y,\rho)$.  
Then the sequence $(\Phi(x_{n}))_{n\in\mathbb{N}}$ is a Cauchy sequence in $(Y,\rho)$.  
Since $(Y,\rho)$ is complete, the sequence $(\Phi(x_{n}))_{n\in\mathbb{N}}$ converges to a point $c\in Y$.  All points of the sequence $(\Phi(x_{n}))_{n\in\mathbb{N}}$  are pairwise distinct because the points of 
$(x_{n})_{n\in\mathbb{N}}$ are pairwise distinct.
Consequently the point $c$ is a limit point of $Y$.  Now using Lemma~\ref{gfjhkg} we obtain 
$
(Y,\rho)\in {\bf US}.
$
Thus  $\Phi:X\to Y$ is the desired isometric embedding of $(X,d)$ into some {\bf US}-space.

$(i_2)$ Let $(i_2)$ hold. Then \((X,d)\) 
contains two distinct points 
 $x_0$ and $y_0$ such that
\begin{equation}
    \label{qq1}
d(x_0,y_0) \leq d(y_1,y_2)
\end{equation}
 for all distinct $y_1,\,y_2\in X$. Now, using Lemma~\ref{fint}, we see that  $(X,d)$ is an $\bf US$-space. Thus statement $(i)$ hold with $(Y,\rho)= (X,d)$.

To complete the proof we need to consider  case  $(i_3)$.

$(i_3)$ Let $(i_3)$ hold. Then the mapping $d^0:X\times X\to\mathbb{R}^+$ defined as 
\begin{equation}
    \label{kijhw}
d^0(x,y)=
\begin{cases}
d(x,y)-\Delta,& {\rm if} \quad x\neq y,\\[4pt]
0,& {\rm otherwise}
\end{cases}
\end{equation}
is an 
 ultrametric on $X$.

Indeed, using $(i_3)$ and \eqref{kijhw} it is easy to see that
\[
d^{0}(x,y) > 0 \quad \text{and} \quad d^{0}(x,y) = d^{0}(y,x)
\]
hold for all distinct $x,y \in X$.  The strong triangle inequality 
\begin{equation}
    \label{aahnh}
    d^0(x,y)\leq \max \{d^0(x,z),d^0(z,y)\}
\end{equation}
directly follows from \eqref{kijhw} if $x=y$. Let $x$ and $y$ be distinct.

Since $d$ is an ultrametric on $X$, we have the strong triangle inequality
\begin{equation}
    \label{trre}
d(x,y) \leq \max\{ d(x,z),\, d(z,y) \}
\end{equation}
for all $x,y,z \in X$. The last inequality is equivalent to
\begin{equation}\label{dcftyh}
d(x,y) - \Delta \leq \max\{ d(x,z),\, d(z,y) \} - \Delta. 
\end{equation}
The relations  
\[
\max \{ d(x,z),\, d(z,y)\} - \Delta 
= \max \{ d(x,z) - \Delta,\, d(z,y) - \Delta \}
\]
\[
\leq\max \{ d^0(x,z),\, d^0(z,y) \}
\]
and \eqref{trre}, \eqref{dcftyh} give us \eqref{aahnh}.
Thus $(X,d^0)$ is an ultrametric space as required.

It directly follows from \eqref{ety1}, \eqref{ety2}, and \eqref{kijhw} that
\begin{equation}
    \label{w1}
\inf \{ d^0(x,y) : x,y \in X,\, x \neq y \} = 0. 
\end{equation}

 Definition~\ref{alfth} and $(i_3)$ imply that $(X,d)$ is a discrete ultrametric space. Since $(X,d)$ is a discrete ultrametric space, $(i_3)$ and \eqref{kijhw}  imply that 
the ultrametric $(X,d^0)$ also is discrete.
Now using \eqref{w1} and Proposition \ref{qsswd} we see that the discrete space $(X,d^{0})$ is not metrically discrete.  

Let $A$ be an arbitrary subset of $X$ with $|A|=4$, and let $G_{A,d_A}$ and $G_{A,d^{0}_A}$ be the diametrical graphs of the ultrametric spaces $(A, d\big|_{A\times A})$ and $(A, d^{0}\big|_{A\times A}),$ respectively.  
Using \eqref{kijhw} it is easy to see that the identical mapping 
\(
\mathrm{Id}_A : A \to A, \quad \mathrm{Id}_A(a) = a
\)
for each $a \in A$, is an isomorphism of the diametrical graphs $G_{A,d_A}$ and $G_{A,d^{0}_A}$.
Hence, by statement $(ii)$, \((X,d^0)\) 
contains no four-point subspace whose diametrical graph is isomorphic to the cycle $C_4$. 
The last statement, $(iii)$ and Lemma~\ref{pidscy} imply that 
\((X,d^0)\) contains no four-point subspace which is weakly similar to $(X_4,d_4)$ or to $(Y_4,\rho_4)$.
Now, arguing as in case $(i_1)$, we can prove the existence of 
an {\bf US}-space \((Y,\rho^0)\) such that \((X,d^0)\) is isometrically embeddable in \((Y,\rho^0)\).

Let \(\Phi^0 : X \to Y\) be an isometric embedding of \((X,d^0)\) in \((Y,\rho^0)\) and let the mapping 
$
\rho : Y \times Y \to \mathbb{R} 
$
be defined as
\begin{equation}
    \label{wsknb}
\rho(x,y) := 
\begin{cases}
\rho^0(x,y)+\Delta, & {\rm if}\quad x\neq y,\\
0, & {\rm otherwise}.
\end{cases}
\end{equation}
As in the case of transition \(d \to d^0\), it is easy to verify that \(\rho : Y \times Y \to \mathbb{R}^+\) 
is an ultrametric on \(Y\). Since $(Y,\rho^0)$ is an ${\bf US}$-space, Theorem \ref{[2.1]} and \eqref{wsknb} imply the membership relation $(Y,\rho)\in {\bf US}$. Moreover, the restriction of the ultrametric \(\rho \) on the set \(\Phi^0(X)\times \Phi^0(X)\), where 
\[
\Phi^0(X) := \{\Phi^0(x), \, x \in X\},
\]
gives us an ultrametric space that is isometric to the space \((X,d)\).
Thus the mapping
\[
(X,d) \xrightarrow{\operatorname{Id}_X} (X,d^0) 
\xrightarrow{\Phi^0} (Y,\rho^0) 
\xrightarrow{\operatorname{Id}_Y} (Y,\rho),
\]
where \(\operatorname{Id}_X\) and \(\operatorname{Id}_Y\) are, respectively, the identical mappings on the sets
\(X\) and \(Y\), 
is the desired isometric embedding of \((X,d)\) into {\bf US}-space $(Y,\rho)$.

Thus statement $(i)$ holds in all cases. 
The proof is completed.

\end{proof}

Theorem \ref{eeerlm}, Theorem \ref{8866gh} and Theorem \ref{mnb} give us the next result.

\begin{corollary}
    \label{eehgg}
Let $(X,d)$ be an infinite ultrametric space. 
Then $(X,d)$ can be isometrically embedded in an {\bf US}-space
if and only if every four-point subspace of $(X,d)$ can be
isometrically embedded in an ${\bf US}$-space.
\end{corollary}

\section{Conclusion. Expected results}

Theorem \ref{thm:3.1} and Theorem \ref{mnb}  imply that an infinite ultrametric space $(X,d)$ 
admits an isometric embedding in an ${\bf US}$-space 
if and only if for each four-point $A \subseteq X,$ the diametrical 
graph $G_{A,d_A}$ is isomorphic to one of the graphs 
$K_{1,1,1},$ $K_{1,1,2}$ or $K_{1,3}.$
If $G_{A,d_A} $ and $K_{1,1,1}$ are isomorphic 
for every four-point $A \subseteq X,$ then it is easy to prove that 
$(X,d)$ is equidistant, i.e. there is $t>0$ such that
\[
d(x,y) = t 
\]
for all distinct $ x,y \in X.$

Thus we obtain the next simple proposition.

\begin{proposition}
\label{gjjt}
Let $(X,d)$ be an infinite  ultrametric space.  
Then the following statements are equivalent:

\begin{itemize}[left=10pt]

    \item[(i)] $(X,d)$ is equidistant.
    \item[(ii)] The diametrical graph of every four-point subspace of $(X,d)$ is isomorphic to $K_{1,1,1,1}$.
\end{itemize}
\end{proposition}

Let us now turn to graph $K_{1,1,2}$.

\begin{example}
    
\end{example}

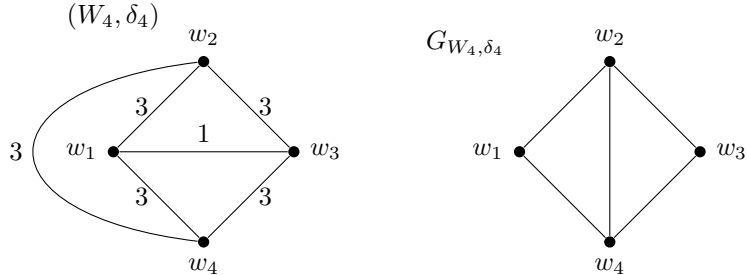
\begin{figure}[ht!]
\centering
\begin{tikzpicture}[remember picture, scale=1.2, every node/.style={font=\small}]

  \begin{scope}[xshift=-3.5cm]
    \node at (-1.0,1.5) {$(W_4,\delta_4)$};
    \node[circle,fill,inner sep=1.5pt,label=above:$w_2$] (x2) at (0,1) {};
    \node[circle,fill,inner sep=1.5pt,label=left:$w_1$]  (x1) at (-1,0) {};
    \node[circle,fill,inner sep=1.5pt,label=right:$w_3$] (x3) at (1,0) {};
    \node[circle,fill,inner sep=1.5pt,label=below:$w_4$] (x4) at (0,-1) {};
    \draw (x1) -- node[left] {$3$} (x2);
    \draw (x2) -- node[right,yshift=-0pt] {$3$} (x3);
    \draw (x3) -- node[right] {$3$} (x4);
    \draw (x4) -- node[left] {$3$} (x1);
    \draw (x1) -- node[above] {$1$} (x3);
    \draw (x4) .. controls (-2.5,-0.7) and (-2.5,0.7) .. node[left] {$3$} (x2);
  \end{scope}

\begin{scope}[xshift=1.0cm]
    \node at (-1.6,1.2) {$G_{W_4,\delta_4}$};
    \node[circle,fill,inner sep=1.5pt,label=above:$w_2$] (y2) at (0,1) {};
    \node[circle,fill,inner sep=1.5pt,label=left:$w_1$]  (y1) at (-1,0) {};
    \node[circle,fill,inner sep=1.5pt,label=right:$w_3$] (y3) at (1,0) {};
    \node[circle,fill,inner sep=1.5pt,label=below:$w_4$] (y4) at (0,-1) {};
    \draw (y1) -- (y2) -- (y3) -- (y4) -- (y1);
    \draw (y2) -- (y4); %
\end{scope}

\end{tikzpicture}
\caption{The four-point ultrametric spaces $(W_4,\delta_4)$ and its diametrical graph $G_{W_4,\delta_4}$.}
\label{cis5}
\end{figure}

We hope that the following conjecture is true.

\begin{conjecture}
Let $(X,d)$ be an   ultrametric space with $|X|\geq 4$. Then the following statements are equivalent:

\begin{itemize}[left=10pt]
    \item[(i)] The diametrical graph of each four-point subspace of $(X,d)$ is isomorphic to $K_{1,1,2}$.
    \item[(ii)] Each four-point subspace of $(X,d)$ is weakly isometric to $(W_4,\delta_4)$.
    \item[(iii)] The space $(X,d)$ is weakly isometric to $(W_4,\delta_4)$.
\end{itemize}

\end{conjecture}

Let us consider now two examples of ${\bf US}$-spaces whose four-point subspaces have diametrical graph isomorphic to $K_{1,3}$.

\begin{example}
 \label{fgjkdfb}
Let us define an ultrametric $d^+\colon \mathbb{R}^+ \times \mathbb{R}^+ \to \mathbb{R}^+$ as
\begin{equation}\label{reew}
d^+(p, q) =
\begin{cases}
0, & \text{if } \quad p = q, \\
\max \{p, q\}, & \text{if }\quad  p \neq q.
\end{cases}
\end{equation}
In \cite{DR2025USGbLSG} it was noted that $({\mathbb R}^+,d^+)\in {\bf US}$. Moreover using \eqref{reew} it is easy to see that the diametrical graphs of all four-point subspaces of $(\mathbb R^+,d^+)$ are isomorphic to $K_{1,3}.$

\end{example}

\begin{example}
    See Figure \ref{cis6} below.
\end{example}

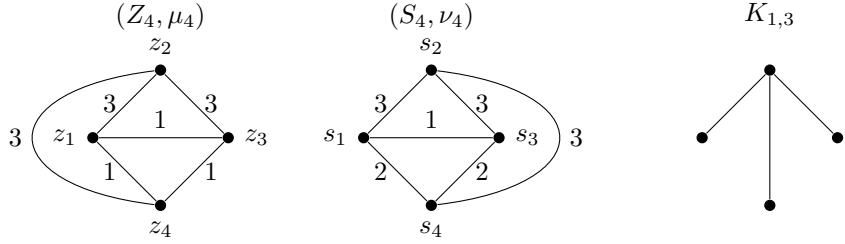
\begin{figure}[ht!]
\centering
\begin{tikzpicture}[remember picture, scale=0.9, every node/.style={font=\small}]

  \begin{scope}[xshift=-7cm]
    \node at (0,1.8) {$(Z_4,\mu_4)$};
    \node[circle,fill,inner sep=1.5pt,label=above:$z_2$] (x2) at (0,1) {};
    \node[circle,fill,inner sep=1.5pt,label=left:$z_1$]  (x1) at (-1,0) {};
    \node[circle,fill,inner sep=1.5pt,label=right:$z_3$] (x3) at (1,0) {};
    \node[circle,fill,inner sep=1.5pt,label=below:$z_4$] (x4) at (0,-1) {};
    \draw (x1) -- node[left] {$3$} (x2);
    \draw (x2) -- node[right] {$3$} (x3);
    \draw (x3) -- node[right] {$1$} (x4);
    \draw (x4) -- node[left] {$1$} (x1);
    \draw (x1) -- node[above] {$1$} (x3);
    \draw (x4) .. controls (-2.5,-0.7) and (-2.5,0.7) .. node[left] {$3$} (x2);
  \end{scope}

  \begin{scope}[xshift=-3cm]
    \node at (0,1.8) {$(S_4,\nu_4)$};
    \node[circle,fill,inner sep=1.5pt,label=above:$s_2$] (y2) at (0,1) {};
    \node[circle,fill,inner sep=1.5pt,label=left:$s_1$]  (y1) at (-1,0) {};
    \node[circle,fill,inner sep=1.5pt,label=right:$s_3$] (y3) at (1,0) {};
    \node[circle,fill,inner sep=1.5pt,label=below:$s_4$] (y4) at (0,-1) {};
    \draw (y1) -- node[left] {$3$} (y2);
    \draw (y2) -- node[right] {$3$} (y3);
    \draw (y3) -- node[right] {$2$} (y4);
    \draw (y4) -- node[left] {$2$} (y1);
    \draw (y1) -- node[above] {$1$} (y3);
    \draw (y2) .. controls (2.5,0.7) and (2.5,-0.7) .. node[right] {$3$} (y4);
  \end{scope}

  \begin{scope}[xshift=2cm]
    \node at (0,1.8) {$K_{1,3}$};
    \node[circle,fill,inner sep=1.5pt,label=above:] (z2) at (0,1) {};
    \node[circle,fill,inner sep=1.5pt,label=left:]  (z1) at (-1,0) {};
    \node[circle,fill,inner sep=1.5pt,label=right:] (z3) at (1,0) {};
    \node[circle,fill,inner sep=1.5pt,label=below:] (z4) at (0,-1) {};
    \draw (z1) -- (z2) -- (z3);
    \draw (z2) -- (z4); 
  \end{scope}

\end{tikzpicture}
\caption{The diametrical graphs of the four-point ultrametric spaces $(S_4,\nu_4)$ and $(Z_4,\mu_4)$ are isomorphic to $K_{1,3}$.}
\label{cis6}
\end{figure}

The following hypothesis is very plausible.

\begin{conjecture}
    \label{gdgnij}
    Let $(X,d)$ be an   ultrametric space with $|X|\geq 4$. Then the following statements are equivalent:

\begin{itemize}[left=10pt]
    \item[(i)] The diametrical graph of each four-point subspace of $(X,d)$ is isomorphic to $K_{1,3}$, but
$(X,d)$ contains no four-point subspaces which are weakly similar to $(Z_4,\mu_4)$.

\item[(ii)] All four-point subspace of $(X,d)$ are weakly similar to $(S_4,\nu_4)$.

\item[(iii)] The space $(X,d)$ admits an isometric embedding into $(\mathbb R^+,d^+)$.
    \end{itemize}
\end{conjecture}

The ultrametric $d^+$ on $\mathbb{R}^+$ was introduced by Delhommé, Laflamme, Pouzet, and Sauer in \cite{DLPS2008TaiA}.
Some results related to the ultrametric space $ (\mathbb{R}^+, d^+) $ can be found in \cite{DK2024JMS,Dov2025Arx,Dov2024UPFAME,Isha2023,Isha2021,Isha2023-1}.

\section*{Declarations}

\subsection*{Declaration of competing interest}

 The authors declare no conflict of interest.

\subsection*{Data availability}

 All necessary data are included into the paper.

\subsection*{Funding}

First author was supported by grant $359772$ of the Academy of Finland.\\


\bibliographystyle{plainurl}
\bibliography{Forbidden}

\bigskip

CONTACT INFORMATION

\medskip
Oleksiy Dovgoshey\\
Department of Function Theory, Institute of Applied Mathematics and Mechanics of NASU, Slovyansk, Ukraine,\\
Department of Mathematics and Statistics, University of Turku, Turku, Finland \\
oleksiy.dovgoshey@gmail.com, oleksiy.dovgoshey@utu.fi

\medskip
Olga Rovenska\\
Department of Mathematics and Modelling, Donbas State Engineering Academy, Kramatorsk, Ukraine\\
rovenskaya.olga.math@gmail.com

\end{document}